\documentclass[11pt]{article}
\usepackage{amscd}
\usepackage{amsfonts}
\usepackage{amsmath}
\usepackage{amssymb}
\usepackage{amsthm}
\usepackage{bbm}
\usepackage{CJK}
\usepackage{fancyhdr}
\usepackage{graphicx}
\usepackage{indentfirst}
\usepackage{latexsym}
\usepackage{mathrsfs}
\usepackage{tikz-cd}
\usepackage{xcolor} 

\def\a{\alpha}
\def\b{\beta}
\def\r{\rho}

\def\k{\kappa}

\def\o{\otimes}

\allowdisplaybreaks[4]
\newtheorem{theorem}{Theorem}[section]
\newtheorem{lemma}[theorem]{Lemma}
\newtheorem{definition}[theorem]{Definition}
\newtheorem{proposition}[theorem]{Proposition}

\newtheorem{corollary}[theorem]{Corollary}
\newtheorem{remark}[theorem]{Remark}
\newtheorem{hypothesis}[theorem]{Hypothesis}

\usepackage[top=1in,bottom=1in,left=1.25in,right=1.25in]{geometry}
\date{}
\begin{document}
\renewcommand{\baselinestretch}{1.2}
\renewcommand{\arraystretch}{1.0}
\title{\bf  Hochschild cohomology and AS-Gorenstein property of weak Hopf Galois extensions}
\date{}
\author {{\bf Daowei Lu$^1$\footnote {Corresponding author:  ludaowei620@126.com}, Dingguo Wang$^{2,3}$}\\
{\small $^1$School of Mathematics and Big Data, Jining University}\\
{\small Qufu, Shandong 273155, P. R. China}\\
{\small $^2$Department of General Education, Shandong Xiehe University}\\
{\small Jinan, Shandong, 250109, P. R. China}\\
{\small $^3$School of Mathematical Sciences, Qufu Normal University}\\
{\small Qufu, Shandong 273165, P. R. China}
}
 \maketitle

\begin{center}
\begin{minipage}{12.cm}

\noindent{\bf Abstract.} Let $H$ be a weak Hopf algebra with bijective antipode. This paper is devoted to the AS-Gorenstein property of restricted faithfully flat weak $H$-Galois extensions $A/B$. We first of all study the cohomologies of weak $H$-Galois extensions and then establish the spectral sequence connecting the Hochschild cohomologies of $A$ and of $B$. Finally we prove that when $A$ is a noetherian affine PI algebra and $B$ is AS-Gorenstein, $A$ shares the AS-Gorenstein property.
\\

\\

 \end{minipage}
 \end{center}
 \normalsize\vskip1cm
 
 \section*{Introduction}

Let $H$ be a Hopf algebra, and $A$ a right $H$-comodule algebra. Denote $B=A^{coH}$. Then we say $A/B$ is an $H$-Galois extension if there exists a bijection $\textup{can}:A\o_BA\rightarrow A\o H$ defined by $\textup{can}(x\o y)=xy_{(0)}\o y_{(1)}$. The Hopf Galois extension plays an essential role in the structures of Hopf algebras and has deep connections with crossed products, cleft extensions and the category of relative Hopf modules \cite{Mon}. It provides a unifying framework for the study of Galois extensions of fields and rings, strongly graded algebras and affine algebraic principal homogeneous spaces. The are numerous literatures on the research of Hopf Galois extension, see for example \cite{BH,D,DT, JS,Schau96,Schau97, Sch}.

D. Stefan \cite{Ste} studied Hochschild cohomology on Hopf Galois extensions and proved that there exists a spectral sequence $H^p(H,H^q(B,M))\Longrightarrow H^{p+q}(A,M)$, which connected the cohomologies of $A$ and of $B$. Resorting to this spectral sequence, Zhu \cite{Zhu} studied the AS-Gorenstein property of Hopf Galois extension and established that when $A$ is a noetherian affine PI algebra and $B$ is AS-Gorenstein, $A$ inherits the AS-Gorenstein property.
 
Weak Hopf algebras were first introduced by Böhm and Szlach$\acute{a}$nyi \cite{Bohm99} as an extension of both classical Hopf algebras and groupoid algebras. Structurally, a weak Hopf algebra consists of a vector space endowed with compatible algebra and coalgebra operations that satisfy a self-dual consistency condition, along with an antipode map. What distinguishes weak Hopf algebras from ordinary ones is that the comultiplication is not required to preserve the unit--equivalently, the counit need not be an algebra homomorphism. 
The significance of the research of weak Hopf algebras lies in its connection with the theory of algebra extensions and a natural framework for the study of dynamical twists in Hopf algebras. Moreover, from the viewpoint of category, any semisimple monoidal category with finitely many simple objects is equivalent to the representation category of a weak Hopf algebra \cite{Os}. 

Numerous fundamental properties inherent to ordinary Hopf algebras have been found to possess counterparts within the weak Hopf algebra framework. For instance, the integral theory for weak Hopf algebras, as formulated in \cite{Bohm99}, exhibits a structural parallelism with the classical theory. It is widely recognized that the homological integral, which generalizes the classical integral from finite-dimensional Hopf algebras to the infinite dimensional setting, has proven instrumental in the classification of infinite dimensional Hopf algebras with low Gelfand-Kirillov dimension \cite{LWZ}. More recently, this homological notion has been successfully extended to the weak Hopf algebra setting \cite{RWZ26}. Another significant advance in this area concerns weak Hopf Galois extensions $A/B$, first introduced in \cite{Caen2007}. This concept has subsequently been employed to derive a structure theorem for endomorphism algebras over weak Doi-Hopf modules \cite{NWZ}, as well as to explore the homological dimension relationship between the algebras $A$ and $B$ \cite{Zhou}.

Inspired by these results, in this paper we aim to construct spectral sequences of the cohomologies of weak Hopf Galois extension and investigate their AS-Gorenstein property, thereby extending the main results of \cite{Ste} and \cite{Zhu}.

The paper is organized as follows. Section 1 recalls the basic definitions and results concerning weak Hopf algebras and weak Hopf Galois extensions. In Section 2, we establish an isomorphism of Ext-groups that will serve as a key ingredient in the following derivation. Let $H$ be a weak Hopf algebra and $A/B$ a restricted flat weak $H$-Galois extension. In Section 3, we firstly define a natural $H$-action on $H^\bullet(B,M)$ for $A^e$-module $M$, and then prove the existence of spectral sequences of the cohomologies of weak Hopf Galois extension. Several applications of these spectral sequences are also provided. In Section 4, we turn our attention to the AS-Gorenstein property of weak Hopf Galois extensions, and prove that any restricted faithfully flat Hopf Galois extension of an AS-Gorenstein algebra remains AS-Gorenstein.

 \section{Preliminaries}
 \def\theequation{1.\arabic{equation}}
\setcounter{equation} {0}

Throughout this article, let $k$ be a fixed field, and all vector spaces and tensor product are taken over $k$ unless otherwise stated. For a coalgebra $C$, we will use the Heyneman-Sweedler's notation $\Delta(c)=  c_{1}\otimes c_{2},$ for any $c\in C$ (summation omitted).

Let $H$ be a linear space with the structures of an associative algebra $(H,m,1)$ and a coassociative coalgebra $(H,\Delta,\varepsilon)$. Recall from \cite{Bohm99}, $H$ is a weak Hopf algebra if it satisfies the following conditions
\begin{itemize}
  \item [(1)] The comultiplication $\Delta$ is a (not necessarily unit-preserving) homomorphism of algebras such that
  \begin{align*}
  &(\Delta\o id)\Delta(1)=(\Delta(1)\o1)(1\o\Delta(1)),\\
&(\Delta\o id)\Delta(1)=(1\o\Delta(1))(\Delta(1)\o1).
  \end{align*}
  \item [(2)] The counit $\varepsilon$ satisfies the identity
  \begin{equation*}
\varepsilon(xyz)=\varepsilon(xy_1)\varepsilon(y_2z)=\varepsilon(xy_2)\varepsilon(y_1z).
\end{equation*}
  \item [(3)] There exists an algebra and coalgebra anti-homomorphism $S:H\rightarrow H$, called an antipode, satisfying
\begin{align*}
& x_1S(x_2)=\varepsilon(1_1x)1_2,\\
&S(x_1)x_2=1_1\varepsilon(x 1_2),\\
& S(x_1)x_2S(x_3)=S(x),
\end{align*}
\end{itemize}
for all $x,y,z\in H$. 

For any weak Hopf algebra, define linear maps $\varepsilon_t,\varepsilon_s:H\rightarrow H$ by
$$\varepsilon_t(x)=\varepsilon(1_1x)1_2,\quad \varepsilon_s(x)=1_1\varepsilon(x 1_2),$$ 
where $\Delta(1)=\sum1_1\o1_2$. Denote the images of $\varepsilon_t$ and $\varepsilon_s$ by $H_t$ and $H_s$ respectively. According to \cite[Proposition 2.11]{Bohm99}, $H_t$ and $H_s$ are semisimple algebras. And $H_s$ is a right $H$-module with the action $y\cdot h=\varepsilon_s(yh)$ for $y\in H_s,h\in H$.

Now we list some identities used in this paper as follows:
\begin{align}
&\Delta(x)=1_1x\o1_2,\ \Delta(y)=1_1\o y1_2,\ \forall x\in H_t,y\in H_s,\label{1a}\\
&h_1\o\varepsilon_s(h_2)=h1_1 \o S(1_2),\ \varepsilon_t(h_1)\o h_2=S(1_1)\o 1_2 h,\ \forall h\in H,\label{1b}\\
&h_1y\o h_2=h_1\o h_2S(y),\ \forall h\in H,y\in H_s.\label{1c}
\end{align}

For left $H$-modules $M,N$, define
$$M\widehat{\o}N=\left\{\sum_i 1_1\cdot m_i\o 1_2\cdot n_i\mid \sum_i m_i\o n_i\in M\o N\right\}.$$
Then $M\widehat{\o}N$ is a left $H$-module with the action $h\cdot(m\widehat{\o}n)=h_1\cdot m\widehat{\o} h_2\cdot n$. Similarly for right $H$-modules $M,N$, define
$$M\overline{\o}N=\left\{\sum_im_i\cdot 1_1 \o n_i \cdot 1_2\mid \sum_i m_i\o n_i\in M\o N\right\}.$$
Then $M\overline{\o}N$ is a right $H$-module with the action $(m\overline{\o}n)\cdot h=m\cdot h_1\overline{\o}n\cdot h_2$. Following the arguments in \cite{RWZ}, we obtain that there are natural identifications $M\widehat{\o}N=M\o_{H_t}N$ as $H_t$-bimodules and $M\overline{\o}N=M\o_{H_s}N$ as $H_s$-bimodules.

Let $H$ be a weak Hopf algebra. An associative algebra $A$ is called a right $H$-comodule algebra if $A$ is a right $H$-comodule via $\r:A\rightarrow A\o H,\ a\mapsto a_{(0)}\o a_{(1)}$, and for all $a,b\in A$,
\begin{align}
&(\r\o id)\circ\r=(id\o\Delta)\circ\r,\label{1.2a}\\
&1_{(0)}a \o 1_{(1)}=a_{(0)}\o\varepsilon_t(a_{(1)}),\label{1.2b}\\
&\r(ab)=\r(a)\r(b),\label{1.2c}
\end{align}
where $\r(1)=\sum 1_{(0)}\o 1_{(1)}$. It is noticed that the identity (\ref{1.2b}) has an equivalent form
\begin{equation}
1_{(0)}\o1_{(1)1}\o1_{(1)2}=1_{(0)}\o1_11_{(1)}\o1_{2}.\label{1.2b'}
\end{equation}

If $A$ is a right $H$-comodule  algebra, the vector subspace of $A$
$$A^{coH}=\left\{a\in A|\r(a)=a_{(0)}\o\varepsilon_t(a_{(1)})\right\}$$
is called the $H$-coinvariants of $A$.

Let $A$ be an $H$-comodule algebra. A vector space $M$ is called an $(A, H)$-Hopf module if $M$ is a right $A$-module and a right $H$-comodule such that
$$(m\cdot a)_{(0)}\o (m\cdot a)_{(1)}=m_{(0)}\cdot a_{(0)}\o m_{(1)}a_{(1)},$$
for all $a\in A,m\in M$. The category of $(A,H)$-Hopf modules is denoted by $\mathcal{M}^H_A$, with the morphism being right $A$-linear and right $H$-colinear.

\begin{lemma}\cite{Zhang04}
Let $H$ be a weak Hopf algebra, $A$ a right $H$-comodule algebra, and $M$ an $(A,H)$-Hopf module. Then we have
\begin{itemize}
  \item [(1)] for all $m\in M,x\in H$,
  \begin{equation}
  m_{(0)}\varepsilon(m_{(1)}x)=m\cdot 1_{(0)}\varepsilon(1_{(1)}x).\label{1.2e}
  \end{equation}
  \item [(2)] $M^{coH}=\left\{m\in M|\rho(m)=m_{(0)}\o \varepsilon_t(m_{(1)})\right\}=\left\{m\in M|\rho(m)=m\cdot1_{(0)}\o 1_{(1)}\right\}$.
  \item [(3)] $A^{coH}=\left\{a\in A|\r(a)=a1_{(0)}\o1_{(1)}\right\}$.
\end{itemize}
\end{lemma}
 Therefore for $a\in A^{coH}$, we have $\r(a)=a1_{(0)}\o 1_{(1)}=1_{(0)}a\o 1_{(1)}$, and $A^{coH}$ is a subalgebra of $A$.

 Let $A$ be a right $H$-comodule algebra. Set 
 $$A\overline{\o} H=\left\{\sum _i a_i1_{(0)}\o h_i1_{(1)}|\sum _i a_i\o h_i\in A\o H\right\},$$
 which is a subspace of $A\o H$. 
 
Recall from \cite{Caen2007} that $A$ is a weak $H$-Galois extension of $B=A^{coH}$ if the canonical map can: $A\o _BA\rightarrow A\overline{\o}H$ given by 
 $$\textup{can}(a\o_B a')=aa'_{(0)}\overline{\o }a'_{(1)},\ a,a'\in A,$$
 is a bijection.
 
We also have a pair of adjoint functors $(F,G)$ defined by
$$F:\mathcal{M}_B\rightarrow\mathcal{M}^H_A,\ N\mapsto N\o_BA;\quad G:\mathcal{M}^H_A\rightarrow\mathcal{M}_B,\ M\mapsto M^{coH}.$$
 \begin{proposition}\cite[Proposition 2.3]{Caen2007}
Let $H$ be a weak bialgebra, and $A$ a right $H$-comodule algebra. Then the following assertions are equivalent:
\begin{itemize}
  \item [(1)] $A$ is a weak $H$-Galois extension of $B$ and faithfully flat as a left $B$-module;
  \item [(2)] $(F,G)$ is an equivalence and $A$ is flat as a left $B$-module.
\end{itemize}
 \end{proposition}
 We say a weak Hopf Galois extension $A/B$ is flat if $A$ is flat both as a left and right $B$-module, and $A/B$ is faithfully flat if $A$ is faithfully flat as a left or right $B$-module. 

According to \cite{NWZ}, the {\it translation map} associated to the weak $H$-Galois extension $A/B$ is given as follows
$$\k: H\rightarrow A\o_BA,\quad h\mapsto \textup{can}^{-1}(1_{(0)}\o h1_{(1)}).$$
In what follows, for any $h\in H$ we use the notation
$$\k(h)=\sum\k^1(h)\o_B \k^2(h).$$
Thus $\textup{can}(\k(h))=\sum\k^1(h)\k^2(h)_{(0)}\o \k^2(h)_{(1)}=1_{(0)}\o h1_{(1)}$.

\begin{lemma}\cite{NWZ}\label{lem:2a}
For all $a\in A,g,h\in H$, we have the following identities
\begin{align}
&\sum\k^1(gh)\o_B \k^2(gh)=\sum\k^1(h)\k^1(g)\o_B \k^2(g)\k^2(h),\label{2a}\\
&\sum b\k^1(h)\o_B \k^2(h)=\sum \k^1(h)\o_B \k^2(h)b,\ b\in B,\label{2b}\\
&\sum\k^1(h)\o_B \k^2(h)_{(0)}\o \k^2(h)_{(1)}=\sum\k^1(h_1)\o_B \k^2(h_1)\o h_2,\label{2c}\\
&\sum\k^1(h)_{(0)}\o_B \k^2(h)\o \k^1(h)_{(1)}=\sum\k^1(h_2)\o_B \k^2(h_2)\o S(h_1),\\
&\sum \k^1(h)\k^2(h)=\varepsilon(h1_{(1)})1_{(0)},\label{2d}\\
&\sum a_{(0)}\k^1(a_{(1)})\o_B \k^2(a_{(1)})=1\o_B a.\label{2e}
\end{align}
\end{lemma}

In addition, it is straightforward to obtain  
\begin{equation}
\textup{can}^{-1}(a1_{(0)}\o h1_{(1)})=\sum a\k^1(h)\o_B \k^2(h).\label{2f}
\end{equation}

Let $M$ be a left module over an algebra $A$. We write $\textup{p.dim}(_AM)$ and $\textup{i.dim}(_AM)$ for the projective and injective dimensions of $M$, respectively. The left and right global dimensions of $A$ are denoted by $\textup{lgl.dim}(A)$ and $\textup{rgl.dim}(A)$.

The algebra $A$ is said to have finite injective dimension if the injective dimensions of the left module $_AA$  and the right module $A_A$ are both finite and equal; this common value is denoted by $\textup{i.dim}(A)$.

Similarly, $A$ is said to have finite global dimension if its left and right global dimensions are finite and coincide; this common value is denoted by $\textup{gl.dim}(A)$.

Let $R$ be a ring. Recall from \cite{Zhu}, an $R$-module $M$ is said to be of type $\mathbf{FP}_n$ ($n \geq 0$) if there is a projective resolution
$$ \cdots\rightarrow P_{n+1}\rightarrow P_n\rightarrow \cdots\rightarrow P_1\rightarrow P_0\rightarrow M\rightarrow 0, $$
where $P_0, P_1, \dots, P_n$ are finitely generated.
Furthermore, $M$ is said to be of type $\mathbf{FP}_{\infty}$ if each $P_n$ in the above projective resolution is finitely generated for all integers $n \geq 0$.
An $R$-module $M$ of type $\mathbf{FP_{\infty}}$ is said to be of type $\mathbf{FP}$, if $M$ has finite projective dimension.
 
\section{An isomorphism between Ext-groups}
 \def\theequation{2.\arabic{equation}}
\setcounter{equation} {0}

Let $H$ be a weak Hopf algebra. In this section, we mainly construct an isomorphism between Ext-groups which will be used later.
  
\begin{lemma}\label{lem:2l}
Let $H$ be a weak Hopf algebra with bijective antipode. 
\begin{itemize}
  \item [(1)] For right $H$-modules $V,W,X$, we have a linear isomorphism
$$\textup{Hom}_{H^{op}}(V\overline{\o} W, X)\cong\textup{Hom}_{H^{op}}(V,\textup{Hom}_{H^{op}_s}(W, X)),$$
which is natural in $V,W,X$.
  \item [(2)] For left $H$-modules $V,W,X$, we have a linear isomorphism
$$\textup{Hom}_H(V\widehat{\o} W, X)\cong\textup{Hom}_H(V,\textup{Hom}_{H_s}(W, X)),$$
which is natural in $V,W,X$.
\end{itemize}

\end{lemma}

\begin{proof}
(1) First of all, $\textup{Hom}_{H^{op}_s}(W, X)$ is a right $H$-module under the action
$$(f\cdot h)(w)=f(w\cdot S^{-1}(h_2))\cdot h_1,$$
for all $f\in\textup{Hom}_{H^{op}_s}(W, X),h\in H,w\in W$. For $x\in H_s$,
\begin{align*}
(f\cdot h)(w\cdot x)&=f(w\cdot xS^{-1}(h_2))\cdot h_1\\
&=f(w\cdot S^{-1}(h_2)S(x))\cdot h_1\\
&\stackrel{(\ref{1c})}{=}f(w\cdot S^{-1}(h_2))\cdot h_1x\\
&=(f\cdot h)(w)\cdot x,
\end{align*}
hence $f\cdot h\in\textup{Hom}_{H^{op}_s}(W, X)$. And it is a routine exercise to check that is action is associative and unital. 

Now define a linear map $\Gamma:\textup{Hom}_{H^{op}}(V,\textup{Hom}_{H^{op}_s}(W, X))\rightarrow \textup{Hom}_{H^{op}}(V\overline{\o} W, X)$ by
$$\Gamma(f)(v\cdot 1_1\o w\cdot 1_2)=f(v\cdot 1_1)(w\cdot 1_2),$$
for $f\in\textup{Hom}_{H^{op}}(V,\textup{Hom}_{H^{op}_s}(W, X)),v\in V,w\in W.$ For $h\in H$, we have
\begin{align*}
&\Gamma(f)((v\cdot 1_1\o w\cdot 1_2)\cdot h)=\Gamma(f)(v\cdot h_1\o w\cdot h_2)\\
&=f(v\cdot h_1)(w\cdot h_2)=(f(v)\cdot h_1)(w\cdot h_2)\\
&=f(v)(w\cdot h_3S^{-1}(h_2))\cdot h_1=f(v)(w\cdot S^{-1}(h_2S(h_3)))\cdot h_1\\
&=f(v)(w\cdot S^{-1}(\varepsilon_t(h_2)))\cdot h_1=f(v)(w)\cdot S^{-1}(\varepsilon_t(h_2))h_1\\
&=f(v)(w)\cdot h,
\end{align*}
thus $\Gamma(f)$ is right $H$-linear and $\Gamma$ is well defined. 

Then we define a linear map $\Lambda: \textup{Hom}_{H^{op}}(V\overline{\o} W, X)\rightarrow\textup{Hom}_{H^{op}}(V,\textup{Hom}_{H^{op}_s}(W, X))$ by
$$\Lambda(g)(v)(w)=g(v\cdot 1_1\o w\cdot 1_2),$$
for $g\in\textup{Hom}_{H^{op}}(V\overline{\o} W, X),v\in V,w\in W$. For $h\in H,x\in H_s$,
\begin{align*}
\Lambda(g)(v)(w\cdot x)&=g(v\cdot 1_1\o w\cdot x1_2)\\
&=g(v\cdot x_1\o w\cdot x_2)\\
&=g(v\cdot 1_1\o w\cdot 1_2)\cdot x,
\end{align*}
and
\begin{align*}
&(\Lambda(g)(v)\cdot h)(w)=\Lambda(g)(v)(w\cdot S^{-1}(h))\\
&=g(v\cdot 1_1\o w\cdot S^{-1}(h_2)1_2)\cdot h_1=g(v\cdot h_1\o w\cdot S^{-1}(h_3) h_2)\\
&=g(v\cdot h_1\o w\cdot S^{-1}(\varepsilon_s(h_2)))\stackrel{(\ref{1c})}{=}g(v\cdot h1_1\o w\cdot 1_2)\\
&=\Lambda(g)(v\cdot h)(w),
\end{align*}
which means that $\Lambda$ is well defined. Finally since
\begin{align*}
&\Lambda(\Gamma(f))(v)(w)=\Gamma(f)(v\cdot 1_1\o w\cdot 1_2)\\
&=f(v\cdot 1_1)(w\cdot 1_2)=(f(v)\cdot 1_1)(w\cdot 1_2)\\
&=f(v)(w\cdot 1_3S^{-1}(1_2))\cdot 1_1=f(v)(w\cdot S^{-1}(\varepsilon_t(1_2)))\cdot 1_1\\
&=f(v)(w)\cdot S^{-1}(\varepsilon_t(1_2))1_1=f(v)(w),
\end{align*}
and 
\begin{align*}
\Gamma(\Lambda(g))(v\cdot 1_1\o w\cdot 1_2)&=\Lambda(g)(v\cdot 1_1)(w\cdot 1_2)=g(v\cdot 1_1\o w\cdot 1_2),
\end{align*}
we conclude that $\Gamma$ and $\Lambda$ are isomorphisms. 

(2) For a weak Hopf algebra $H$ and left $H$-modules $M,N$, $\textup{Hom}_{H_s}(M,N)$ is a left $H$-module under the action
$$(h\cdot f)(m)=h_1\cdot f(S(h_2)\cdot m),$$
for all $h\in H,f\in\textup{Hom}_{H_t}(M,N).$ The second isomorphism could be verified similarly.
The proof is completed.
\end{proof}

\begin{remark}
From the above lemma, 

(1) for right $H$-modules $M,N$, if $N$ is projective, so is $M\overline{\o} N$.

(2) for left $H$-modules $M,N$, if $M$ is projective, so is $M\widehat{\o} N$.
\end{remark}

\begin{proposition}\label{pro:2m}
Let $H$ be a weak Hopf algebra.

(1) For all $i\geq0$ and right $H$-modules $V,W, X$, we have 
$$\textup{Ext}^i_{H^{op}}(V\overline{\o} W, X)\cong\textup{Ext}^i_{H^{op}}(V,\textup{Hom}_{H^{op}_s}(W, X)).$$

(2) For all $i\geq0$ and left $H$-modules $V,W, X$, we have 
$$\textup{Ext}^i_{H}(V\widehat{\o} W, X)\cong\textup{Ext}^i_{H}(V,\textup{Hom}_{H_s}(W, X)).$$
\end{proposition}

\begin{proof}
We only prove the first assertion and the second could be verified similarly.
Let $P_\bullet$ be an $H$-projective resolution of $V$. Then
$$\cdots\rightarrow P_i\overline{\o} W\rightarrow\cdots \rightarrow P_1\overline{\o} W \rightarrow P_0\overline{\o} W \rightarrow V\overline{\o} W\rightarrow0$$
is an $H$-projective resolution for $V\overline{\o} W$ by \cite[Lemma 6.4(3)]{RWZ}. By Lemma \ref{lem:2l} we get a commutative diagram of complexes as follows:
$$\begin{CD}
\textup{Hom}_{H^{op}}(P_{0},\textup{Hom}_{H^{op}_s}(W,X)) @> >>\cdots @> >>\textup{Hom}(P_{i},\textup{Hom}_{H^{op}_s}(W,X))\\
@VV\cong V @. @VV\cong V \\
\textup{Hom}_{H^{op}}(P_{0}\overline{\o} W,X) @> >> \cdots @> >>\textup{Hom}_{H^{op}}(P_{i}\overline{\o} W,X)
\end{CD}$$
The homology of the top row is $\textup{Ext}^\bullet_{H^{op}}(V,\textup{Hom}_{H^{op}_s}(W, X))$, while the homology of the bottom is $\textup{Ext}^\bullet_{H^{op}}(V\overline{\o} W, X)$. The proof is completed.
\end{proof}

The following two corollaries follow directly from the proof of  \cite[Corollary 1.4]{BG} and left to the reader.
\begin{corollary}\label{coro:2n}
Let $H$ be a weak Hopf algebra and $V,W,X$ are right $H$-modules. Then
\begin{itemize}
  \item [(1)] $\textup{p.dim}_{H^{op}}(V\overline{\o }W)\leq \textup{p.dim}_{H^{op}}(V)$.
  \item [(2)] $\textup{i.dim}_{H^{op}}(\textup{Hom}_{H^{op}_s}(V,X))\leq \textup{i.dim}_{H^{op}}(X)$.
  \item [(3)] $\textup{rgl.dim}(H)=\textup{p.dim}(_{H^{op}}H_s)$.
\end{itemize}
\end{corollary}

 \begin{corollary}
 Let $H$ be a weak Hopf algebra and $V,W,X$ are left $H$-modules. Then
 \begin{itemize}
  \item [(1)] $\textup{p.dim}_H(V\widehat{\o}W)\leq \textup{p.dim}_H(V)$.
  \item [(2)] $\textup{i.dim}_{H}(\textup{Hom}_{H_s}(W,X))\leq \textup{i.dim}_{H}(X)$.
  \item [(3)] $\textup{lgl.dim}(H)=\textup{p.dim}(_{H}H_t)$.
\end{itemize}
 \end{corollary}

 \section{Cohomology on weak Hopf Galois extensions}
 \def\theequation{3.\arabic{equation}}
\setcounter{equation} {0}

Let $B=A^{coH}\subseteq A$ be a weak $H$-Galois extension. In this section, we construct spectral sequences relating the Hochschild cohomologies of $A$ and $B$, extending the main results of \cite{Ste} to the weak Hopf algebraic framework.

For an algebra $A$, we write $A^e=A\o A^{op}$ for its enveloping algebra, so that an $A$-$A$-bimodule $M$ could  be identified as a left $A^e$-module. Let $H^n(A,-)$ be the $n$-th right derived functor of $\textup{Hom}_{A^e}(A,-)$ in the category of $A^e$-modules. The spaces $H^\bullet(A,M)$ are called the Hochschild cohomology groups of $A$ with coefficients in $M$. For more details one can refer to \cite{Ste}.

 For any $B^e$-module $M$, define
$$M^B=\{m\in M\mid b\cdot m=m\cdot b, \forall b\in B\},\quad M_B=M/[M,B],$$
where $[M,B]$ is the subspace of $M$ generated by all commutators $[m,b]=m\cdot b-b\cdot m$.
 
It is worthwhile to point out that $\k(H)$ lies in $(A\o_BA)^B$ from the identity (\ref{2b}) and $\k$ could be modified as
$$\k: H\rightarrow (A\o_BA)^B.$$

\begin{lemma}\label{lem:2b}
Let $H$ be a weak Hopf algebra and $A$ a right $H$-comodule algebra. 

(a) For any $A^e$-module $M$, $M^B$ is a right $H$-module with the following action
$$m\leftharpoonup h=\sum \k^1(h)\cdot m\cdot \k^2(h),$$
for all $m\in M^B,h\in H$.

(b) For any $A^e$-module $M$, $M_B$ is a left $H$-module with the following action
$$h\rightharpoonup(m+[M,B])=\sum \k^2(h)\cdot m\cdot \k^1(h)+[M,B],$$
for all $m\in M,h\in H$.
\end{lemma}
 
\begin{proof} 
(a) Actually by (\ref{2b}) and \cite[Lemma 4.1]{Ste} this action is well defined, independent of the choice of $\k^1(h)$ and $\k^2(h)$. And by (\ref{2a}) we have that the action is associative. Since $\k(1_H)=1_A\o_B1_A$, the action is unital.

(b) First of all the left action of $H$ on $M_B$ is well defined. Indeed for all $h\in H, b\in B,m\in M$, 
\begin{align*}
h\rightharpoonup(bm+[M,B])&=\sum \k^2(h)bm\k^1(h)+[M,B]\\
&\stackrel{(\ref{2b})}{=}\sum \k^2(h)m b\k^1(h)+[M,B]\\
&=h\rightharpoonup(mb+[M,B]).
\end{align*}
Then it is easy to check that the action is associative and unital. The proof is completed.
 \end{proof}
 

\begin{definition}
Let $A/B$ be a weak $H$-Galois extension. We say this extension is restricted if the image $\k(H_s)\subseteq B\o_B B$.
\end{definition}

\begin{remark}
Let $H$ be a weak Hopf algebra. Then by \cite[Proposition 2.7]{Caen2007} $H$ is a weak $H$-Galois extension of $H_t$, where $H$ is a right $H$-comodule algebra via $\Delta$ and the inverse of  the Galois map $\textup{can}$ is given by
$$\textup{can}^{-1}:H\overline{\o }H \rightarrow H\o _{H_t}H,\ 1_1\o h1_2\mapsto\varepsilon_t(1_1)S(h_1)\o h_21_2.$$ 
Now suppose that $\Delta(1)=\Delta^{op}(1)$, then $H_t=H_s$.
For $h\in H$,
\begin{align*}
\k(\varepsilon_s(h))&=can^{-1}(1_1\o \varepsilon_s(h)1_2)\\
&=\varepsilon_t(1_1)S(\varepsilon_s(h)_1)\o_{H_t} \varepsilon_s(h)_21_2\\
&=\varepsilon_t(1_1)S(1'_1)\o_{H_t} \varepsilon_s(h)1'_21_2\\
&=S(1'_1 1_1)\o_{H_t} \varepsilon_s(h)1'_21_2\\
&=S(1_1)\o_{H_t} \varepsilon_s(h)1_2\\
&=S(1_1)\o_{H_t} 1_2\varepsilon_s(h)\\
&=S(1_1)1_2\o_{H_t}\varepsilon_s(h)\\
&=1\o_{H_t}\varepsilon_s(h)\in H_s\o_{H_t}H_s= H_t\o_{H_t}H_t.
\end{align*}
\end{remark}
 
\begin{lemma}\label{lem:2c}
Let $A/B$ be a restricted flat weak $H$-Galois extension. Then for an $A^e$-module $M$, there exists a natural right action of $H$ on $H^0(B,M)$ such that 
$$\textup{Hom}_{H^{op}}(H_s,H^0(B,M))\cong H^0(A,M).$$
\end{lemma}
 
 \begin{proof}
 
For $f\in \textup{Hom}_{H^{op}}(H_s,M^B)$ and $h\in H$, we have
\begin{align*}
f(1)\leftharpoonup h=f(1\cdot h)=f(\varepsilon_s(h))=f(\varepsilon^2_s(h))=f(1)\leftharpoonup \varepsilon_s(h),
\end{align*}
and for $m\in M^B$ with $m\leftharpoonup h=m\leftharpoonup\varepsilon_s(h)$, set $f_m(z)=m\leftharpoonup z$ for $z\in H_s$. Then by (\ref{2b}) we have $f_m(z)\in M^B$, and for all $h\in H$,
$$f_m(z\cdot h)=f_m(\varepsilon_s(zh))=m\leftharpoonup \varepsilon_s(zh)=m\leftharpoonup zh=f_m(z)\leftharpoonup h,$$
that is, $f_m\in\textup{Hom}_{H^{op}}(H_s,M^B)$. It is straightforward to verify that the above correspondences are bijective, hence we obtain $$\textup{Hom}_{H^{op}}(H_s,M^B)\cong\left\{m\in M^B\mid m\leftharpoonup h=m\leftharpoonup\varepsilon_s(h),\forall h\in H\right\}.$$

Now for $m\in M^A$ and $h\in H$, 
\begin{align*}
m\leftharpoonup h&=\sum \k^1(h)\cdot m\cdot\k^2(h)=\sum m\cdot \k^1(h)\k^2(h)\\
&=m\cdot1_{(0)}\varepsilon(h1_{(1)})=m\cdot1_{(0)}\varepsilon(1_11_{(1)})\varepsilon(h1_2)\\
&=m\leftharpoonup\varepsilon_s(h).
\end{align*}
Conversely if $m\leftharpoonup h=m\leftharpoonup\varepsilon_s(h)$ for $m\in M^B,h\in H$, then for $a\in A,$
 \begin{align*}
 m\cdot a&\stackrel{(\ref{2e})}{=}\sum a_{(0)}\k^1(a_{(1)})\cdot m\cdot\k^2(a_{(1)})\\
 &=a_{(0)}\cdot(m\leftharpoonup a_{(1)})\\
 &=a_{(0)}\cdot(m\leftharpoonup \varepsilon_s(a_{(1)}))\\
 &=\sum a_{(0)}\k^1(\varepsilon_s(a_{(1)}))\cdot m\cdot\k^2(\varepsilon_s(a_{(1)}))\\
 &=\sum a_{(0)}\k^1(\varepsilon_s(a_{(1)}))\k^2(\varepsilon_s(a_{(1)}))\cdot m\\
 &\stackrel{(\ref{2d})}{=} a_{(0)}1_{(0)}\varepsilon(\varepsilon_s(a_{(1)})1_{(1)})\cdot m\\
 &=a_{(0)}1_{(0)}\varepsilon(a_{(1)}1_{(1)})\cdot m\\
 &=a\cdot m,
 \end{align*}
 which implies that $M^A=\left\{m\in M^B\mid m\leftharpoonup h=m\leftharpoonup\varepsilon_s(h),\forall h\in H\right\}$.
 
 Since $H^0(B,M)=\textup{Hom}_{B^e}(B,M)\cong M^B$, we have 
 $$\textup{Hom}_{H^{op}}(H_s,H^0(B,M))\cong\textup{Hom}_H(H_s,M^B)\cong M^A=H^0(A,M).$$
 
 Let $M,N$ be two $A^e$-modules and $f:M\rightarrow N$ a morphism of $A^e$-modules. Obviously the induced map $f^0:M^A\rightarrow N^A$ is $H$-linear. Thus the $H$-module structure is natural. The proof is completed.
 \end{proof}

Let $R:\! _A\mathcal{M}_A\rightarrow\! _B\mathcal{M}_B$ be the functor of restriction of scalars.

\begin{lemma}\cite{Ste}\label{lem:2d}
If $A$ is flat as left and right $B$-module, then the functor $H^\bullet(B,-)\circ R:\! _A\mathcal{M}_A\rightarrow \mathcal{M}_k$ is cohomological and coeffaceable.
\end{lemma}

\begin{proposition}\label{pro:2e}
For any $q\geq0$ there exists a right action of $H$ on $H^q(B,M)$.
\end{proposition}

\begin{proof}
The proof of \cite[Proposition 2.4]{Ste} is applicable without changes. Here we will give a brief construction. For any $ h \in H $ let $\rho_0^h(-) : H^0(B, -) \circ R \to H^0(B, -) \circ R$ be the natural morphism given by
$$ \rho_0^h(M)(m) = m \cdot h,$$
for $m\in H^0(B, M)$ and $h\in H$.  By Lemma \ref{lem:2d}, $\rho_0^h(-)$ could be extended to 
$$\rho_{\bullet}^h(-) : H^\bullet(B, -) \circ R \to H^\bullet(B, -) \circ R,$$ 
and $\rho_{\bullet}^h(-)$ defines a natural $H$-module structure on $H^q(B, -) \circ R$.  In addition, any $ f \in \textup{Hom}_{A^e}(M, N) $ induces an $H$-linear map $f^q : H^q(B, M) \to H^q(B, N)$. The proof is completed.
\end{proof}

Let $M$ be an $A^e$-module, then Hom$(A^e,M)$ is a left $A^e$-module with the module structure coming form the right $A^e$-module $A^e$. In detail, for $a,b,x,y\in A$ and $f\in\textup{Hom}(A^e,M)$,
$$((a\o b)\cdot f)(x\o y)=f((x\o y)(a\o b))=f(xa\o by).$$

\begin{lemma}\label{lem:2f}
Let $A/B$ be a flat weak $H$-Galois extension. For  $A^e$-module $M$, we have an isomorphism of right $H$-modules 
$$\textup{Hom}(A^e,M)^B\cong \textup{Hom}(H\overline{\o}A,M).$$
\end{lemma}

\begin{proof}
First of all $\textup{Hom}(A^e,M)^B=H^0(B,\textup{Hom}(A^e,M))$ is a right $H$-module by Proposition \ref{lem:2c}. Explicitly for $f\in \textup{Hom}(A^e,M)^B,h\in H$,
$$(f\leftharpoonup h)(x\o y)=\sum f(x\k^1(h)\o \k^2(h)y).$$
As shown in \cite[Lemma 3.1]{Ste}, there is an isomorphism 
$$\a:\textup{Hom}(A^e,M)^B\rightarrow \textup{Hom}(A\o_B A,M), f\mapsto\a(f),\ \a(f)(x\o_B y)=f(x\o y).$$
By this identification $\textup{Hom}(A\o_B A,M)$ becomes a right $H$-module with the action given by
\begin{equation}\label{2g}
(f\leftharpoonup h)(x\o_B y)=\sum f(x\k^1(h)\o_B \k^2(h)y),
\end{equation}
and thus $\a$ is an isomorphism of $H$-modules. Apply the contravariant functor $\textup{Hom}(-,M)$ to the morphisms can$^{-1}$ and we get the isomorphism 
$$\textup{can}^{-1*}:\textup{Hom}(A\o_B A,M)\rightarrow \textup{Hom}( A\overline{\o}H,M).$$
Note that $\textup{Hom}(A\overline{\o} H,M)$ is a right $H$-module since $A\o H$ is a left $H$-module with the action on the second tensorand. For $f\in\textup{Hom}(A\o_B A,M),a\in A,g,h\in H$,
\begin{align*}
(\textup{can}^{-1*}(f)\cdot g)(a1_{(0)}\o h1_{(1)})&=\textup{can}^{-1*}(f)(a1_{(0)}\o gh1_{(1)})\\
&=f(\textup{can}^{-1*}(a1_{(0)}\o gh1_{(1)}))\\
&\stackrel{(\ref{2f})}{=}f\left(\sum a\k^1(gh)\o_B\k^2(gh)\right)\\
&\stackrel{(\ref{2a})}{=}f\left(\sum a\k^1(h)\k^1(g)\o_B\k^2(g)\k^2(h)\right)\\
&\stackrel{(\ref{2g})}{=}(f\leftharpoonup g)\left(\sum a\k^1(h)\o_B\k^2(h)\right)\\
&=((f\leftharpoonup g)\circ \textup{can}^{-1})(a1_{(0)}\o h1_{(1)})\\
&=\textup{can}^{-1*}(f\leftharpoonup g)(a1_{(0)}\o h1_{(1)}).
\end{align*}
Let $\sigma$ be the flip map, then $\sigma( A\overline{\o}H)=H\overline{\o} A$. Using the isomorphism $\textup{Hom}( A\overline{\o}H,M)\cong \textup{Hom}(H\overline{\o}A,M)$ of right $H$-modules, the proof is completed.
\end{proof}

\begin{proposition}\label{pro:2g}
Let $A/B$ be a restricted flat weak $H$-Galois extension. If $M$ is an injective $A^e$-module, then $\textup{Ext}^q_{H^{op}}(H_s,M^B)=0$ for all $q\geq1$.
\end{proposition}

\begin{proof}
Define a linear map $i:M\rightarrow\textup{Hom}(A^e,M)$ by $i(m)(x\o y)=x\cdot m\cdot y$. It is easy to check that $i$ is an injective morphism of $A^e$-modules. The fact that $M$ is an injective $A^e$-module guarantees the following decomposition  
$$\textup{Hom}(A^e,M)=M\oplus N,$$
as $A^e$-modules. Hence as right $H$-modules, we have 
$$\textup{Hom}(A^e,M)^B=M^B\oplus N^B.$$
Since the functor Ext preserves direct sums it suffices to show 
$$\textup{Ext}^q_{H^{op}}(H_s,\textup{Hom}(A^e,M)^B)=0.$$
Let $P_\bullet$ be an $H$-projective resolution of $H_s$ as right $H$-module. For any $n\geq1$, denote $P_n\overline{\o} A=\{p\cdot1_{(1)}\o a1_{(0)}\mid p\in P_n,a\in A\}$. Then by Lemma \ref{lem:2f},  $\textup{Ext}^q_{H^{op}}(H_s,\textup{Hom}(A^e,M)^B)$ is the $q$th cohomology group of the complex
$$\textup{Hom}_{H^{op}}(P_\bullet,\textup{Hom}(H\overline{\o}A,M))\cong\textup{Hom}(P_\bullet\o_HH\overline{\o}A,M)=\textup{Hom}(P_\bullet\overline{\o}A,M).$$
For any $n\geq1$ and the embedding map $i_n:P_n\overline{\o}A\hookrightarrow P_n\o A$, apply the functor $\textup{Hom}(-,M)$ to $i_n$ and we get an epimorphism 
$$i^*_n:\textup{Hom}(P_n\o A,M)\rightarrow\textup{Hom}(P_n\overline{\o}A,M).$$
Define a linear map $j^*_n:\textup{Hom}(P_n\overline{\o}A,M)\rightarrow\textup{Hom}(P_n\o A,M)$ by
$$j^*_n(f)(p\o a)=f(p\cdot 1_{(1)}\o a1_{(0)}),\quad f\in \textup{Hom}(P_n\overline{\o}A,M).$$
Then it is easy to check that $i^*_n\circ j^*_n$ is the identity map on $\textup{Hom}(P_n\overline{\o}A,M)$.

Consider the following diagram
$$\begin{CD}
\textup{Hom}(P_{n-1}\overline{\o}A,M) @>(d_n\o id)^*>> \textup{Hom}(P_n\overline{\o}A,M)\\
@VVj^*_{n-1}V @VVj^*_nV\\
\textup{Hom}(P_{n-1}\o A,M) @>(d_n\o id)^*>> \textup{Hom}(P_n\o A,M)\\
@VVi^*_{n-1}V @VVi^*_nV\\
\textup{Hom}(P_{n-1}\overline{\o}A,M) @>(d_n\o id)^*>> \textup{Hom}(P_n\overline{\o}A,M)
\end{CD}$$
For $\a\in\textup{Hom}(P_{n-1}\overline{\o}A,M)$ and $p\in P_n,a\in A$, 
\begin{align*}
(j^*_n\circ(d_n\o id)^*)(\a)(p\o a)&=j^*_n(((d_n\o id)^*)(\a))(p\o a)\\
&=((d_n\o id)^*)(\a))(p\cdot1_{(1)}\o a1_{(0)})\\
&=\a(d_n(p\cdot1_{(1)})\o a1_{(0)})\\
&=\a(d_n(p)\cdot1_{(1)}\o a1_{(0)}),
\end{align*}
and 
\begin{align*}
((d_n\o id)^*\circ j^*_{n-1})(\a)(p\o a)&=(d_n\o id)^*(j^*_{n-1}(\a))(p\o a)\\
&=j^*_{n-1}(\a)(d_n(p)\o a)\\
&=\a(d_n(p)\cdot1_{(1)}\o a1_{(0)}).
\end{align*}
Therefore $j^*_n\circ(d_n\o id)^*=(d_n\o id)^*\circ j^*_{n-1}$, and the upper block is commutative.

For $\b\in\textup{Hom}(P_{n-1}\o A,M)$,
\begin{align*}
&(i^*_n\circ(d_n\o id)^*)(\b)(p\cdot1_{(1)}\o a1_{(0)})\\
&=i^*_n(((d_n\o id)^*)(\b))(p\cdot1_{(1)}\o a1_{(0)})\\
&=((d_n\o id)^*)(\b))(p\cdot1_{(1)}\o a1_{(0)})\\
&=\b(d_n(p\cdot1_{(1)})\o a1_{(0)})\\
&=\b(d_n(p)\cdot1_{(1)}\o a1_{(0)}),
\end{align*}
and
\begin{align*}
&((d_n\o id)^*\circ i^*_{n-1})(\b)(p\cdot1_{(1)}\o a1_{(0)})\\
&=(d_n\o id)^*(i^*_{n-1}(\b))(p\cdot1_{(1)}\o a1_{(0)})\\
&=i^*_{n-1}(\b)(d_n(p)\cdot1_{(1)}\o a1_{(0)})\\
&=\b(d_n(p)\cdot1_{(1)}\o a1_{(0)}).
\end{align*}
Therefore  $i^*_n\circ(d_n\o id)^*=(d_n\o id)^*\circ i^*_{n-1}$, and the lower block is commutative. Hence $i^*=\{i^*_n\}$ and $j^*=\{j^*_n\}$ are chain maps, and $H^n(i^*_n)\circ H^n(j^*_n)=H^n(i^*_n\circ j^*_n)=id$, which means that $H^n(i^*_n)$ is an epimorphism. Since the $k$-complex $\textup{Hom}(P_\bullet\o A,M)$ is exact, $H^q\textup{Hom}(P_\bullet\o A,M)=0$, and thus $H^q\textup{Hom}(P_\bullet\overline{\o}A,M)=0$.
The proof is completed.
 
\end{proof}

\begin{theorem}\label{thm:2h}
Let $A/B$ be a restricted flat weak $H$-Galois extension. Then for an $A^e$-module $M$ there is a spectral sequence
$$\textup{Ext}^p_{H^{op}}(H_s,H^q(B,M))\Longrightarrow H^{p+q}(A,M),$$
which is natural in $M$.
\end{theorem}

\begin{proof}
Let $F,F_1,F_2$ be the following three functors
\begin{align*}
&F:\! _A\mathcal{M}_A\rightarrow \mathcal{M}_k,\quad F(-)=\textup{Hom}_{A^e}(A,-),\\
&F_1:\mathcal{M}_H\rightarrow\mathcal{M}_k,\quad F_1(-)=\textup{Hom}_{H^{op}}(H_s,-),\\
&F_2:\! _A\mathcal{M}_A\rightarrow\mathcal{M}_H,\quad F_2(-)=\textup{Hom}_{B^e}(B,-)\cong(-)^B.
\end{align*}
From Lemma \ref{lem:2c}, $F_1\circ F_2=F$. Proposition \ref{pro:2g} asserts that for any injective $A^e$-module $M$, $F_1(M)$ is $F_2$-acycle. Applying the Grothendieck spectral sequence for the composition of functors yields the convergent spectral sequence
$$(R^pF_1)(R^qF_2(M))\Longrightarrow R^{p+q}F(M),$$
where $R^nF_1,R^nF_2,R^nF$ are the $n$th right derived functor of $F_1,F_2$ and $F$. And by definitions 
$$R^pF_1(-)=\textup{Ext}^p_{H^{op}}(H_s,-),\quad R^pF(-)=H^p(A,-).$$
It follows from the proof of \cite[Lemma 2.2]{Ste} that every injective object in $_A\mathcal{M}_A$ remains injective when regarded as an object of $_B\mathcal{M}_B$. Let $I^\bullet$ be an $A^e$-injective resolution of $M$, then $I^\bullet$ is also a $B^e$-injective resolution of $M$. Consequently $R^qF_2(M)=H^q(B,M).$ The proof is completed.
\end{proof}

Recall that an algebra $A$ is separable if it is projective as $A^e$-module, or equivalently, for any bimodule $M$ and $q\geq1$,
$\textup{Ext}^q_{A^e}(A,M)=H^q(A,M)=0.$

\begin{theorem}\label{thm:2i}
Let $A/B$ be a restricted flat weak $H$-Galois extension. If $H$ is semisimple and $B$ is separable, then $A$ is separable.
\end{theorem}

\begin{proof}
The proof is similar to that of \cite[Theorem 3.6]{Ste} and omitted.
\end{proof}

Let $H$ be a finite dimensional weak Hopf algebra, and $A/B$ a right $H$-Galois extension. Then $A$ is a left $H^*$-module algebra and we obtain the weak smash product algebra $A\# H^*$. According to \cite[Corollary 3.1]{NWZ} there exists an algebra isomorphism 
\begin{equation}\label{2j}
A\# H^*\cong \textup{End}_BA,
\end{equation}
where $A$ is viewed as a right $A$-module, and $A\# H^*/A$ is $H^*$-Galois extension.

\begin{lemma}\label{lem:2k}
Let $H$ be a weak Hopf algebra, and $A$ a left $H$-module algebra.  Let $B=A^H$; then we have an anti-algebra isomorphism $\textup{End}_{A\# H}(A)\cong B$.
\end{lemma}

\begin{proof}
Define $\Phi:\textup{End}_{A\# H}(A)\rightarrow B$ by
$$\Phi(f)=f(1),\quad\forall f\in\textup{End}_{A\# H}(A).$$
For all $h\in H$, $h\cdot f(1)=f(h\cdot 1)=f(\varepsilon_t(h)\cdot 1)=\varepsilon_t(h)\cdot f(1)$, thus $\Phi$ is well defined. For $f,g\in\textup{End}_{A\# H}(A)$,
$$\Phi(fg)=f(g(1))=g(1)f(1)=\Phi(g)\Phi(f),$$
which means that $\Phi$ is an anti-algebra map. Now define $\Psi:B\rightarrow\textup{End}_{A\# H}(A)$ by
$$\Psi(b)(a)=ab,$$
for all $a\in A,b\in B$. Since for $a,a'\in A,h\in H$,
\begin{align*}
(a\# h)\cdot \Psi(b)(a')&=(a\# h)\cdot a'b=a(h\cdot a'b)\\
&=a(h_1\cdot a')(h_2\cdot b)=a(h_1\cdot a')(\varepsilon_t(h_2)\cdot b)\\
&=a(1_1h\cdot a')(1_2\cdot b)=a((h\cdot a')b)\\
&=\Psi(b)((a\# h)\cdot a'),
\end{align*}
that is, $\Psi(b)\in\textup{End}_{A\# H}(A)$ and $\Psi$ is well defined. Finally it is straightforward to check that $\Phi$ and $\Psi$ are mutual inverses. The proof is completed.
\end{proof}

Let $H$ be a weak Hopf algebra and $A$ a left $H$-module algebra. Then $A\# H/A$ is a weak Hopf-Galois extension \cite{NWZ}, with the Galois map
$$\textup{can}:A\# H\o_AA\# H\rightarrow A\# H\overline{\o}H,\ (a\# h)\o_A(b\# g)\mapsto (a\# h)(b\# g_1)\overline{\o} g_2,$$
and its inverse
$$\textup{can}^{-1}:A\# H\overline{\o}H\rightarrow A\# H\o_AA\# H,\ (a\# h)\overline{\o } g\mapsto (a\# h)(1\# S(g_1))\o_A (1\# g_2).$$
Now assume that $\Delta(1)=\Delta^{op}(1)$,  then for all $h\in H$, $\varepsilon_s(h)\in H_s$, and
\begin{align*}
\k(\varepsilon_s(h))&=\textup{can}^{-1}(1\#1_1\overline{\o} \varepsilon_s(h)1_2)\\
&=(1\#1_1)(1\# S(\varepsilon_s(h)_11_2))\o (1\# \varepsilon_s(h)_21_3)\\
&=(1\# 1_1 S(\varepsilon_s(h)_11_2))\o_A (1\# \varepsilon_s(h)_21_3)\\
&=(1\# \varepsilon_t(1_1) S(\varepsilon_s(h)_1))\o_A (1\# \varepsilon_s(h)_21_2)\\
&\stackrel{(\ref{1a})}{=}(1\# \varepsilon_t(1_1) S(1'_1))\o_A (1\# \varepsilon_s(h)1'_21_2)\\
&=(1\# \varepsilon_t(1_1) \varepsilon_t(1'_1))\o_A (1\# \varepsilon_s(h)1'_21_2)\\
&=(1\cdot\varepsilon_t(1_1) \varepsilon_t(1'_1)\#1)\o_A (1\cdot\varepsilon_s(h)1'_21_2\#1).
\end{align*}
Therefore $\k(\varepsilon_s(h))\in A\o_AA,$ namely the extension is restricted.

\begin{corollary}
Let $H$ be a finite dimensional semisimple and cosemisimple weak Hopf algebra such that $\varepsilon(xy)=\varepsilon(yx)$ for $x,y\in H$, and $A/B$ a  restricted flat weak $H$-Galois extension. If $A$ is separable, so is $B$.
\end{corollary}

\begin{proof}
Using $\varepsilon(xy)=\varepsilon(yx)$, we have $\Delta_{H^*}(\varepsilon)=\Delta^{*op}_{H^*}(\varepsilon)$ in $H^*$, and in the $H^*$-Galois extension $A\# H^*/A$, $\k(H^*)\subseteq A\o_AA.$

Since $H$ is cosemisimple, $H^*$ is semisimple. From Theorem \ref{thm:2i} $A\# H^*$ is separable because $A$ is separable. By Lemma \ref{lem:2k}, we have an anti-algebra isomorphism $\textup{End}_{A\# H^*}(A)\cong B$; hence
in order to prove $B$ separable, we suffice to prove $\textup{End}_{A\# H^*}(A)$ separable.

$\textup{End}_{A\# H^*}(A)$ is the commutator of $C=A\# H^*/ann_{A\# H^*} (A)$ in $\textup{End}(A)$. Using \cite[Corollary 3.6]{Zhou} $A\# H^*$ is semisimple. Remake that any separable algebra over a field is finite dimensional \cite[p. 48]{Ingra},  and hence we deduce that $A\# H^*$ is finite dimensional. By a similar argument as in the proof of \cite[Theorem 3.9]{Ste}, we complete the proof.
\end{proof}

By Corollary \ref{coro:2n}, we get the following result immediately.

\begin{corollary}
Let $A/B$ be a restricted flat weak $H$-Galois extension. Then $\textup{p.dim}_{A^e}(A)\leq \textup{p.dim}_H(H_s)+\textup{p.dim}_{B^e}(B)=\textup{rgl.dim}(H)+\textup{p.dim}_{B^e}(B)$.
\end{corollary}

Recall from \cite{Zhu}, let $M$ and $N$ be two right modules over an algebra $R$. Let $P_{\bullet}$ denote a projective resolution of $_{R^e}R$.
Given that the sequence $\cdots \to P_1 \to P_0 \to R \to 0$ is split exact as $R$-module complex, then we obtain the following isomorphism
\begin{equation}
\textup{Ext}^i_{R^e}(R, \textup{Hom}(M, N)) \cong \textup{Ext}^i_{R^{op}}(M, N),\label{2o}
\end{equation}
for all $i \geq 0$.

\begin{lemma}\label{lem:2p}
Let $A/B$ be a restricted flat weak $H$-Galois extension. For any right $A$-modules $M$ and $N$, there is a convergent spectral sequence
$$ \textup{Ext}^{p}_{H^{op}}(H_s, \textup{Ext}^q_{B^{op}}(M, N))\Longrightarrow \textup{Ext}^{p+q}_{A^{op}}(M, N).$$
\end{lemma}

\begin{proof}
Using the isomorphism (\ref{2o}), we have
$$\textup{Ext}^i_{B^e}(B, \textup{Hom}(M, N)) \cong \textup{Ext}^i_{B^{op}}(M, N),\forall i \geq 0,$$
which means that $\textup{Ext}^i_{B^{op}}(M, N)$ is also a right $H$-module via the isomorphism. Therefore by Theorem \ref{thm:2h} there exists a convergent spectral sequence
\begin{align*}
\textup{Ext}^{p}_{H^{op}}&(H_s, \textup{Ext}^q_{B^{op}}(M, N))\cong \textup{Ext}^{p}_{H^{op}}(H_s,\textup{Ext}^q_{B^e}(B, \textup{Hom}(M, N)))\\
 &\Longrightarrow \textup{Ext}^{p+q}_{A^e}(A, \textup{Hom}(M, N)) \cong \textup{Ext}^{p+q}_{A^{op}}(M, N),
 \end{align*}
 as desired.
\end{proof}

\begin{proposition}
Let $A/B$ be a restricted flat weak $H$-Galois extension. Then $\textup{rgl.dim}(B) \leq\textup{rgl.dim}(H) +\textup{rgl.dim}(A)$.
\end{proposition}

\section{AS-Gorenstein property of faithfully flat PI weak Hopf Galois extensions}
 \def\theequation{4.\arabic{equation}}
\setcounter{equation} {0}

\begin{definition}\cite{RWZ}
Let $A$ be a noetherian algebra. We say $A$ is  left AS-Gorenstein (Artin and Schelter Gorenstein) if
\begin{itemize}
  \item [(1)] $A$ has finite injective dimension $d$ as a left and right $A$-module;
  \item [(2)] for every finite dimensional left $A$-module $M$, $\textup{Ext}^i_A(M, A) = 0$ for all $i \neq d$ and $\textup{Ext}^d_A(M, A)$ is a finite-dimensional right module;
  \item [(3)] the right sided analog of (2) also holds.
\end{itemize}
	If further, $\textup{gl.dim}A = d < + \infty$, then $A$ is called {\it AS regular}.
\end{definition}

It is shown in \cite{RWZ} that if $A$ is an AS-Gorenstein algebra of dimension $d$, then $\textup{Ext}^d_A(M, A) \neq 0$ for any finite dimensional $A$-module $M$.

For any left $H$-modules $V$ and $W$, the natural map $\varphi_{W,V}:W\widehat{\o} V^*\rightarrow \textup{Hom}_{H_s}(V,W)$ given by
$$\varphi_{W,V}(w\widehat{\o} f)(v)=f(v)w,\quad w\in W,f\in V^*.$$
It is obvious that $\varphi_{W,V}$ is well defined and left $H$-linear. When $V$ is finite dimensional, define linear map $\psi_{W,V}:\textup{Hom}_{H_s}(V,W)\rightarrow W\widehat{\o} V^*$ by
$$\psi_{W,V}(g)=\sum g(1_1\cdot v_i)\o 1_2\cdot v^i,$$
where $g\in\textup{Hom}_{H_s}(V,W)$, and $\{v_i\}$ and $\{v^i\}$ is basis and dual basis of $V$. A straightforward verification shows that $\varphi_{W,V}$ is an isomorphism with the inverse $\psi_{W,V}$.

For  left $H$-modules $U$ and $W$, $\textup{Hom}_H(U,H)\overline{\o} W$ is a right $H$-module with the module action
$$(f\overline{\o} w)\leftharpoonup h=f\leftharpoonup h_1\overline{\o} S(h_2)\cdot w.$$
Define linear map $\theta:\textup{Hom}_{H}(U, H) \overline{\o}W \rightarrow\textup{Hom}_{H}(U, H\widehat{\o} W)$ by
$$\theta(f \overline{\o} w)(u)= f(u)_{1} \widehat{\o} f(u)_{2}\cdot w.$$

\begin{lemma}\cite[Lemma 3.4, 3.5]{RWZ}\label{lem:4b}
Let $H$ be a weak Hopf algebra. 
\begin{itemize}
  \item [(1)] Suppose that $V,W\in H$-Mod and that $V$ is finite-dimensional. Then, for any $i\geq0$, there is an isomorphism
  $$\textup{Ext}^i_H(W\widehat{\o} V,H)\cong\textup{Ext}^i_H(W,H\widehat{\o} V^*),$$
  as right $H$-modules.
  \item [(2)] Suppose that $U,W\in H$-Mod, and that either $U$ is finite-dimensional or that $U$ has a projective resolution by finitely generated projective modules. Then, for any $i\geq0$,
$$\textup{Ext}^i_H(U,H\widehat{\o} W)\cong \textup{Ext}^i_H(U,H)\overline{\o}W,$$
as right $H$-modules.
\end{itemize}
\end{lemma}

\begin{corollary}\label{cor:4c}
Let $V$ be finite dimensional H-module, and $U$ be an $H$-module of type $FP_\infty$. Then, for any $i\geq0$,
\begin{equation}
\textup{Ext}^i_H(U\widehat{\o}V,H)\cong \textup{Ext}^i_H(U,H)\overline{\o}V^*,\label{4a}
\end{equation}
as right $H$-modules, the right H-module structure on $\textup{Ext}^i_H(U,H)\overline{\o}V^*$ is given by
\begin{equation}
(x\overline{\o} f)\leftharpoonup h=x\leftharpoonup h_1\overline{\o} S(h_2)\cdot f,\label{4b}
\end{equation}
for all $x\in \textup{Ext}^i_H(U,H),f\in V^*,h\in H$.
\end{corollary}

\begin{proof}
By Lemma \ref{lem:4b}, 
$$\textup{Ext}^i_H(U\widehat{\o}V,H)\stackrel{(\ref{4a})}{\cong} \textup{Ext}^i_H(U,H\widehat{\o}V^*)\stackrel{(\ref{4b})}{\cong}\textup{Ext}^i_H(U,H)\overline{\o}V^*.$$
\end{proof}

Let $M$ be a vector space and $N$ be a right $H$-comodule. Then $\textup{Hom}(M, N)$ is a left $H^*$-module, under the action
$$(h^* \rightharpoonup f) (m) = h^* \cdot f(m),$$
for all $h^* \in H^*$, $f \in \textup{Hom}(M, N),m \in M$.

\begin{lemma}\label{lem:4d}
	Let $A/B$ be a weak $H$-Galois extension, $M$ be a right $B$-module and $X \in \mathcal{M}^H_B$.
\begin{itemize}
  \item [(1)] If $M_B$ is finitely generated, then $\textup{Hom}_{B^{op}}(M, X)$ is an $H$-comodule.
  \item [(2)] If $M_B$ is finitely presented, then $\textup{Hom}_{B^{op}}(M, X)^{co H} =\textup{Hom}_{B^{op}}(M, X^{co H})$.
\end{itemize}
\end{lemma}

\begin{proof}
(1) For $f\in\textup{Hom}_{B^{op}}(M, X),m\in M,h^*\in H^*$ and $a\in B$, we have
\begin{align*}
(h^*\rightharpoonup f)(m\cdot a)&=h^*\rightharpoonup(f(m)\cdot a)\\
&=h^*((f(m)\cdot a)_{(1)}) (f(m)\cdot a)_{(0)}\\
&=h^*(f(m)_{(1)}a_{(1)}) f(m)_{(0)}\cdot a_{(0)}\\
&=h^*(f(m)_{(1)}1_{(1)}) f(m)_{(0)}\cdot 1_{(0)}a\\
&=h^*(f(m)_{(1)}) f(m)_{(0)}\cdot a\\
&=(h^*\rightharpoonup f)(m)\cdot a,
\end{align*}
hence $\textup{Hom}_{B^{op}}(M, X)$ is an $H^*$-submodule of $\textup{Hom}(M, X)$.  Then by a similar argument as in the proof of \cite[Lemma 2.5]{Zhu}, we have that $\textup{Hom}_{B^{op}}(M, X)$ is an $H$-comodule.

(2) The proof is similar to that of \cite[Lemma 2.5]{Zhu}. 

The proof is completed.
\end{proof}

\begin{proposition}\label{pro:4e}
Let $M$ be a right $A$-module which is finitely generated over $B$. For any $X\in\mathcal{M}^H_A$, $\textup{Hom}_{B^{op}}(M, X)$ naturally carries a Hopf module structure in $\mathcal{M}^H_H$.
\end{proposition}

\begin{proof}
Define the right $H$-action on $\textup{Hom}_{B^{op}}(M, X)$ by
\begin{equation}
(f \leftharpoonup h) (m) = \sum f(m\cdot\k^1(h))\cdot\k^2(h),
\end{equation}
for all $f \in\textup{Hom}_{B^{op}}(M, X)$, $m \in M$, $h \in H$. It is straightforward to check that $\textup{Hom}_{B^{op}}(M, X)$ is a right $H$-module using (\ref{2a}). Then
\begin{align*}
 &(f \leftharpoonup h)_{(0)} (m) \o (f \leftharpoonup h)_{(1)}\\
 &= (f \leftharpoonup h)(m)_{(0)} \o (f \leftharpoonup h)(m)_{(1)}\\
 &=\sum f(m\cdot\k^1(h))_{(0)} \cdot\k^2(h)_{(0)} \o  f(m\cdot\k^1(h))_{(1)}\k^2(h)_{(1)}\\
 &\stackrel{(\ref{2c})}{=}\sum f (m\cdot\k^1(h_1))_{(0)}\cdot\k^2(h_1)\o f (m\cdot\k^1(h_1))_{(1)} h_2\\
 &=\sum f_{(0)} (m\cdot\k^1(h_1))\cdot\k^2(h_1)\o f_{(1)} h_2\\
 &=(f_{(0)} \leftharpoonup h_1) (m) \o f_{(1)} h_2.
\end{align*}
Therefore $\textup{Hom}_{B^{op}}(M, X)$ belongs to the category $\mathcal{M}^H_H$. The proof is completed.
\end{proof}

\begin{lemma}\label{lem:4f}
Let $M$ be a right $A$-module. If $M_B$ is of type $\mathbf{FP_{\infty}}$, then 
$$\textup{Ext}_{B^{op}}^i(M, A) \cong \textup{Ext}_{B^{op}}^i(M, B) \o_{H_t} H$$ 
as right $H$-modules for $i \geq 0$. Thus $\textup{Ext}_{B^{op}}^i(M, A)$ is a projective right $H$-module.
\end{lemma}

\begin{proof}
For any $X \in \mathcal{M}^H_A$, there is an injective right $A$-module $I$ with an $A$-module inclusion $\iota: X \hookrightarrow I$. Note that $I \overline{\o} H$ is a Hopf module in $\mathcal{M}^H_A$ via
$$(x \overline{\o} h)\cdot a = x\cdot a_{(0)}\overline{ \o} ha_{(1)}, \qquad \rho(x \overline{\o} h) =x \overline{\o} h_1 \o h_2,$$
where $x \in I$, $h \in H$ and $a \in A$. There is a linear embedding $\widetilde{\iota}: X \hookrightarrow I \overline{\o} H, \; x \mapsto \iota(x_{(0)}) \overline{\o} x_{(1)}$. For all $x \in X$ and $a \in A$, we have
\begin{align*}
&\widetilde{\iota}(x\cdot a) =\iota(x_{(0)}\cdot a_{(0)})\overline{ \o} x_{(1)}a_{(1)} = \iota(x_{(0)})\cdot a_{(0)}\overline{ \o} x_{(1)}a_{(1)}=\widetilde{\iota}(x)\cdot a ,\\
&\rho(\widetilde{\iota}(x)) = \rho(\iota(x_{(0)}) \overline{\o} x_{(1)}) =\iota(x_{(0)})\overline{ \o} x_{(1)1} \o x_{(1)2} =\widetilde{\iota}(x_{(0)}) \o x_{(1)},
\end{align*}
which implies that $\widetilde{\iota}$ is a morphism in $\mathcal{M}^H_A$. 

By induction, we could construct an exact complex 
$$0\rightarrow A\rightarrow E^0\rightarrow E^1\rightarrow \cdots$$ 
in $\mathcal{M}^H_A$ where each term $E^i$ is isomorphic to $I^i \overline{\o} H$ and $I^i$ is an injective $A$-module for any $i\geq0$.

Since $_{B}A$ is flat, any injective right $A$-module $I$ is also injective as a $B$-module. Consider the tensor product algebra $A\o H$. Since $1_{(0)}\o 1_{(1)}$ is the idempotent of $A\o H$, then $I\overline{\o} H$ is a direct summand of $I\o H$ as vector spaces, and consequently a direct summand of $I\o H$ as right $B$-module, where $B$ acts on the first tensorand of $I\o H$.

Moreover, it is obvious that $I\o H$ can be seen as a direct sum of copies of $I$ as $B$-modules, and hence $I\overline{\o} H$ can be seen as a direct sum of copies of a direct summand of $I$ as $B$-modules. With the help of \cite[Lemma 2.7]{Zhu}, $I\overline{\o} H$ is $\textup{Hom}_{A^{op}}(M, -)$-acyclic. Thus we obtain that for $i \geq 0$,
$$\textup{Ext}^i_{B^{op}}(M, A) \cong \mathrm{H}^i(\textup{Hom}_{B^{op}}(M, E^{\bullet})),$$
as right $H$-modules.

Since $\textup{Hom}_{B^{op}}(M, E^{\bullet})$ is a complex in $\mathcal{M}^H_H$ by Proposition \ref{pro:4e}, we obtain the following isomorphisms
\begin{align*}
\mathrm{H}^i(\textup{Hom}_{B^{op}}(M, E^{\bullet})) & \cong \mathrm{H}^i(\textup{Hom}_{B^{op}}(M, E^{\bullet}))^{co H} \o_{H_t} H  \\
& \cong \mathrm{H}^i(\textup{Hom}_{B^{op}}(M, E^{\bullet})^{co H}) \o_{H_t} H  \\
& \cong \mathrm{H}^i(\textup{Hom}_{B^{op}}(M, (E^{\bullet})^{co H})) \o_{H_t} H, 
\end{align*}
where the second isomorphism uses the fact that $(-)^{coH}$ is exact and the third uses Lemma \ref{lem:4d}. Since the functor $G$ is an equivalence, the sequence
$$0\rightarrow A = B^{co H}\rightarrow (E^0)^{co H}\rightarrow (E^1)^{co H}\rightarrow \cdots$$
is exact. As $F$ is the quasi-inverse of $G$, we have $(I^i \o_B A)^{co H} \cong I^i$ as $B$-modules.
Consider the composition of the following maps:
$$ I^i \o_B A \stackrel{\cong}{\rightarrow} I^i \o_A (A \o_B A) \stackrel{id_{I^i} \o \beta}{\rightarrow} I^i \o_A (A \overline{\o} H) \stackrel{\cong}{\rightarrow} I^i \overline{\o} H,$$
with the element $x \o_B a \mapsto  x\cdot a_{(0)}\overline{ \o} a_{(1)}$. This composition defines an isomorphism between $I^i \o_B A$ and $I^i \o H$ in $\mathcal{M}^H_A$. Thus, $(E^i)^{co H} = (I^i \overline{\o}H)^{co H} \cong (I^i \o_B A)^{co H} \cong I^i$ as right $B$-modules. Therefore, the complex $(E^\bullet)^{co H}$ constitutes an injective resolution of $B_B$. Then we derive the following isomorphism of right $H$-modules,
\begin{align*}
\textup{Ext}^i_{B^{op}}(M, A) &\cong H^i(\textup{Hom}_{B^{op}}(M, E^{\bullet}))\\
& \cong H^i(\textup{Hom}_{B^{op}}(M, (E^{\bullet})^{co H})) \o_{H_t} H \\
&\cong\textup{Ext}^i_{B^{op}}(M, B) \o_{H_t} H. 
\end{align*}
Since $H_t$ is semisimple, we have that $\textup{Ext}^i_{B^{op}}(M, B) \o_{H_t} H$ is a projective right $H$-module. The proof is completed.
\end{proof}

\begin{hypothesis}\cite{Zhu}\label{hyp:4g}
	Let $R$ be a noetherian affine PI algebra. For any $i \geq 0$, 
	\begin{enumerate}
		\item[(1)] $\textup{Ext}^i_{R^{op}}(U, R) \neq 0$ if and only if $\textup{Ext}^i_{R^{op}}(V, R) \neq 0$ for all simple $R$-modules $U$ and $V$, and 
		\item[(2)] $\textup{Ext}^i_{R^{op}}(-, R)$ is an exact functor on $R$-modules of finite length.
	\end{enumerate}
\end{hypothesis}

In \cite{WuZhang}, Wu and Zhang gave a useful criterion for noetherian affine PI algebras to be AS-Gorenstein.

\begin{theorem}\cite[Corollary 2.10, Propostion 3.2]{WuZhang}\label{thm:4h}
	If the Hypothesis \ref{hyp:4g} holds, then
	\begin{enumerate}
		\item [(1)] $\textup{inj.dim}(_RR) =\textup{inj.dim}(R_R) = d =\textup{Kdim}R < + \infty$;
		\item [(2)]  $R$ is AS-Gorenstein of dimension $d$.
	\end{enumerate}
\end{theorem}

Using the above theorem, we immediately get the main result of this section.

\begin{theorem}
Let $H$ be a weak Hopf algebra, and $A$ is a noetherian affine PI algebra which is a restricted faithfully flat $H$-Galois extension of $B$.
	If $B$ is AS-Gorenstein, then $A$ is also AS-Gorenstein.
	Furthermore if $H$ is AS-Gorenstein of dimension $d_H$, then $\textup{inj.dim}(A_A) = \textup{inj.dim}(B_B) + d_H$. 
\end{theorem}

\begin{proof}
If $A$ is an affine PI algebra, every simple $A$-module is finite dimensional. Suppose that $U$ is a simple right $A$-module, then $U$ is finite dimensional. Since $A$ is faithfully flat as a right $B$-module, we get that $B$ is right noetherian. 	Therefore $U$ is a $B$-module of type $\mathbf{FP_{\infty}}$. Hence we have an isomorphism $\textup{Ext}_{B^{op}}^i(U, A) \cong\textup{Ext}_{B^{op}}^i(U, B) \o_B A$ for all $i \geq 0$. 

If $B$ is AS-Gorenstein of dimension $d_B$, then $\textup{Ext}_{B^{op}}^i(U, B) = 0$ for $i \neq d_B$ and $\textup{Ext}_{B^{op}}^{d_B}(U, B)$ is non-zero and finite dimensional. Hence $\textup{Ext}_{B^{op}}^i(U, A) \neq 0$ if and only if $i = d_B$. By Lemma \ref{lem:2p} we derive that for all $i \geq 0$, 
$$\textup{Ext}_{A^{op}}^i(U, A) \cong \begin{cases}
		0, & i < d_B; \\
		\textup{Ext}_{H^{op}}^{i - d_B}(H_s, \textup{Ext}_{B^{op}}^{d_B}(U, A)), & i \geq d_B. \\
	\end{cases}$$
Using Lemma \ref{lem:4f}, $\textup{Ext}_{B^{op}}^{d_B}(U, A)$ is a projective $H$-module of rank less than dim$(\textup{Ext}_{B^{op}}^{d_B}(U, B))$. Hence we have the equivalences
$$\textup{Ext}_{A^{op}}^i(U, A)\neq0\Leftrightarrow \textup{Ext}_{H^{op}}^{i - d_B}(H_s, \textup{Ext}_{B^{op}}^{d_B}(U, A))\neq0\Leftrightarrow\textup{Ext}_{H^{op}}^{i - d_B}(H_s, H)\neq0.$$
Since for another simple $A$-module $V$,
$$\textup{Ext}_{A^{op}}^i(U, A)\neq0\Leftrightarrow\textup{Ext}_{H^{op}}^{i - d_B}(H_s, H)\neq0\Leftrightarrow\textup{Ext}_{A^{op}}^i(V, A)\neq0,$$
the condition (1) of Hypothesis \ref{hyp:4g} is fulfilled.

Consider an exact sequence $0\rightarrow W'\rightarrow W \rightarrow W'' \rightarrow 0$ of finite length left $A$-modules.
	If $B$ is AS-Groenstein of dimension $d_B$, then we obtain a short exact sequence
	\begin{equation*}
		0\rightarrow \textup{Ext}_{B^{op}}^{d_B}(W'', A)\rightarrow \textup{Ext}_{B^{op}}^{d_B}(W, A)\rightarrow \textup{Ext}_{B^{op}}^{d_B}(W', A)\rightarrow 0
	\end{equation*}
	in $\mathcal{M}^H_H$. By Lemma \ref{lem:4f}, each term is projective as right $H$-modules such that the above exact sequence is split, which means that the following sequence is still exact
$$0\rightarrow\textup{Ext}_{H^{op}}^{i - d_B}(H_s, \textup{Ext}_{B^{op}}^{d_B}(W'', A))\rightarrow\textup{Ext}_{H^{op}}^{i - d_B}(H_s, \textup{Ext}_{B^{op}}^{d_B}(W, A))\rightarrow\textup{Ext}_{H^{op}}^{i - d_B}(H_s, \textup{Ext}_{B^{op}}^{d_B}(W', A))\rightarrow0,$$
therefore we have the following exact sequence
$$0\rightarrow \textup{Ext}_{A^{op}}^{d_A}(W'', A)\rightarrow  \textup{Ext}_{A^{op}}^{d_A}(W, A)\rightarrow  \textup{Ext}_{A^{op}}^{d_A}(W', A)\rightarrow0,$$
that is, the condition (2) of Hypothesis \ref{hyp:4g} is satisfied. Finally by Theorem \ref{thm:4h}, the proof is completed.
\end{proof} 

\section*{Acknowledgement}

This work was supported by the NNSF of China (Nos. 12271292, 11901240).


\end{document}